%% file: main.tex
\documentclass[10pt,reqno]{amsart}
\usepackage{amsmath,amsthm,amssymb}
\usepackage[utf8]{inputenc}

\usepackage{ esint }

\usepackage{amsmath,amsfonts,amssymb,amsthm,mathtools,esint,color,hyperref,cleveref,subcaption}

\usepackage{amsmath,amsfonts,amssymb,amsthm,color}
\usepackage{mathrsfs}

\usepackage[utf8]{inputenc}
\usepackage{graphicx}

\usepackage{marginnote}

\usepackage[inline]{asymptote}

\usepackage{fancybox,color}

\newcommand{\optf}{f^\star}

\usepackage{todonotes}

\newcommand{\bra}[1]{\langle #1 |}
\newcommand{\ket}[1]{| #1 \rangle}
\newcommand{\braket}[2]{\langle #1 | #2 \rangle}

\newtheorem{theorem}{Theorem}[section]

\newtheorem{lemma}[theorem]{Lemma}
\newtheorem{proposition}[theorem]{Proposition}

\newtheorem{definition}[theorem]{Definition}
\newtheorem{remark}[theorem]{Remark}

\newcommand{\R}{{\mathbb R}}

\title{On classical inequalities for autocorrelations and autoconvolutions
}
\author{Jaume de Dios Pont}
\address[JDP]{Department of  Mathematics,  University  of  California  Los  Angeles,  Portola Plaza 520, Los  Angeles,
CA 90095, USA}
\email{jdedios@math.ucla.edu}
\author[Jos\'e Madrid]{Jos\'e Madrid}
\address[JM]{Department of  Mathematics,  University  of  California  Los  Angeles,  Portola Plaza 520, Los  Angeles,
CA 90095, USA}
\email{jmadrid@math.ucla.edu}

\date{\today}
\newtheorem{thm}{Theorem}
\newtheorem{prop}{Proposition} 

\keywords{autocorrelation, autoconvolution, optimal constants, extremizers}
\subjclass[2010]{42A05, 42A85, 70H03, 39A12}

\begin{document}
	\maketitle
	
\begin{abstract}
In this paper we study an autocorrelation inequality proposed by Barnard and Steinerberger \cite{BarnardSteinerberger}. The study of these problems is motivated by a classical problem in additive combinatorics. We establish the existence of extremizers to this inequality, for a general class of weights, including Gaussian functions (as studied by the second author and Ramos) and characteristic function (as originally studied by Barnard and Steinerberger). Moreover, via a discretization argument and numerical analysis, we find some almost optimal approximation for the best constant allowed in this inequality. We also discuss some other related problem about autoconvolutions.

 .
\end{abstract}

\section{Introduction}
\input{intro}

\section{The extremizer exists and is compactly supported}
\input{existence}

\section{Approximation inequalities}
\input{approximation}

\input{computational}

\section{A remark on another autoconvolution problem}
\input{another_problem}

\end{document}

%% file: intro.tex
%!TEX root = main.tex

Motivated by an old problem in additive combinatorics about estimating the size of Sidon sets \cite{CRV}, many authors have studied the problem of finding the best constant $c_{\max}$ such that for any function $f\in L^{1}(\R)$ supported in $[-1/4,1/4]$ the following inequality holds
$$
\text{max}_{-1/2 \le t \le 1/2} \int_{\R} f(t-x)f(x) \, dx \ge c_{\max} \left(\int_{-1/4}^{1/4} f(x) \, dx\right)^2.
$$
The best known result so far was obtained by Cloninger and Steinerberger in \cite{CS}, they proved that $c \ge 1,28.$ Many other lower bounds were previously obtained in \cite{CRT,G,MOB1,Y,MOB2,MV}.
Inspired by this question, two other related problems were proposed and studied by Barnard and Steinerberger in \cite{BarnardSteinerberger}. One of their results was the following: The inequality
\begin{equation}\label{eq:l1l2}
\int_{-1/2}^{1/2}\int_{\R}f(x)f(x+t) dx dt\leq 0.91\|f\|_1\|f\|_2,
\end{equation}
holds for any function $f \in L^1(\R) \cap L^2(\R)$. It was also established in \cite{BarnardSteinerberger} via an example, that the best constant such that \eqref{eq:l1l2} holds is at least 0.8.
The upper bound was recently improved by the second author
and Ramos in \cite{MR}, where they proved that this inequality still hold true when we write 0.865 instead of 0.91. A natural question is, what happen when we consider a different probability space, in particular, what happen when we consider Gaussian means. This question was also addressed in \cite{MR}:
\begin{proposition}[\cite{MR}, Theorem 1.2]\label{Thm JJ}
Let $a$ be a positive real number. For any $f\in L^{1}(\R)\cap L^{2}(\R)$. The following inequality holds
\begin{align}\label{MR thm 1}
\left(\frac{a}{\pi}\right)^{1/2}\int_{\R}\int_{\R}f(x)f(x+t)e^{-a t^{2}}dxdt\leq \left(\frac{8a}{27\pi}\right)^{1/4}\|f\|_1\|f\|_2,
\end{align}
and $\left(\frac{8a}{27\pi}\right)^{1/4}$ can not be replaced by $\left(\frac{a}{4\pi}\right)^{1/4}$. 
\end{proposition}
%In particular, if $a=2\pi$ the upper bound is 0.8773 and the lower bound 0.8408.$.

	%{\it{Structure of the paper:}} 
	In this paper, using Euler-Lagrange equations, in Section 2 we establish the existence of extremizers for \eqref{eq:l1l2}, \eqref{MR thm 1} and a more general class of weights (nonnegative functions $w:\mathbb{R}\to[0,+\infty)$). In Section 3, via a discretization argument, we find an almost optimal numerical approximation for the best constants allowed in these inequalities, see Table 1. Finally, in Section 4, we discuss a related problem about autoconvolution.\\ %Finally, in the appendix, we explain how some extra regularity assumptions on the weight (for instance when dealing with Gaussian functions) allow us to obtain even better approximations (as expected).\\
	
	Related results about autoconvolution inequalities were recently obtained in \cite{CJLL}. We hope variations of the methods outlined in this paper can be applied to such, and other, related problems. As another example,  a classical problem proposed by Erdos, the Minimum overlaping problem, can be reformulated in terms of autorrelations as observed by Haugland in \cite{Ha}, we believe that there is a strong relation between this problem and the one analyzed in this manuscript.

\newcommand{\gr}[1]{{\color{red}#1}}
\begin{table}
\renewcommand{\arraystretch}{1.5}
	\begin{tabular}{ccccc}
		\textbf{Weight} & \multicolumn{3}{c}{\textbf{Spectral Method}} & \textbf{Fixed Point}  \\
			 					& {\small Lower bound} & {\small Upper bound}  & {\small Difference} 	 	& {\small Lower bound} 

		\\
		\hline

		$\chi_{[-1/2,1/2]}$ & $0.80558\gr{09}$ & $0.80558\gr{96}$ & $<9\cdot 10^{-6}$ & $0.80558\gr{09}$\\
		$\exp(-\pi x^2)$    & $0.7152\gr{474}$ & $0.7152\gr{576}$ & $<1.2\cdot 10^{-5}$ & $0.7152\gr{475}$
	\end{tabular}

	\caption{Upper and Lower bounds for the average problem. These bounds have been found with the algorithm described in Section 3.3, with a value of $\delta \approx 1.45\cdot 10^{-3}$ and $\Delta \lambda \approx 0.001$. The implementation can be found at \cite{Github}.}
	\label{table}
\end{table}

\begin{figure}
    \centering
    \includegraphics[width=.9\textwidth]{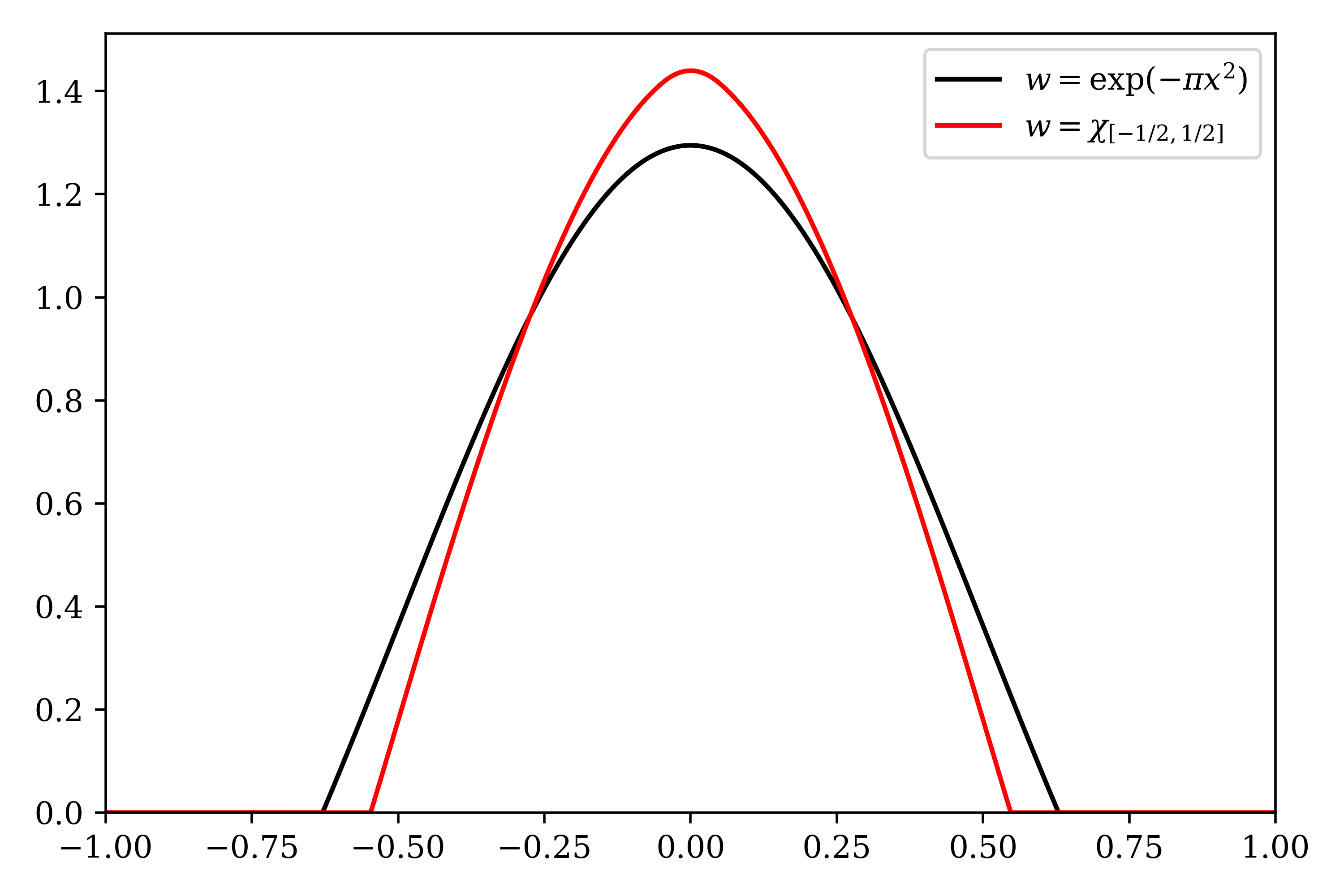}
    \caption{Numerical extremizers for the discretized problem. Extremizers are normalized so that $\|f^*\|_1\|f^*\|_2 = 1$. }
    \label{fig:my_label}
\end{figure}

%% file: existence.tex
%!TEX root = main.tex
	
	Let $w(x) \in L^1(\mathbb{R})\cap L^{\infty}(\mathbb{R})$ be a symmetric decreasing weight with $\|w\|_{L^{1}(\mathbb R)} = \|w\|_{L^{\infty}(\mathbb R)} =1$. Examples to keep in mind are $w(x) = \chi_{[-1/2,1/2]}$, or $w(x) = e^{-\pi x^2}$. Let $C_{opt}$ be the smallest constant such that the inequality 
	\begin{equation}\label{main ineq}
			\int_{\mathbb R^2} f(x)f(y) w(x-y) dx dy \le C_{opt}(w) \|f\|_1\|f\|_2
		\end{equation}
	holds for all $f\in L^{1}(\R)\cap L^{2}(\R)$. After a change of variable, finding $C_{opt}$ for  $\frac{a^{1/2}}{\pi^{1/2}}e^{-a x^2}$ is equivalent to find the optimal constant for which \eqref{MR thm 1} holds. In particular, for $a=\pi$, \eqref{MR thm 1} is equivalent to $0.707107\approx  \frac{1}{2^{1/2}}\le C_{opt}(e^{-\pi x^2}) \le \left(\frac{2}{3}\right)^{3/4}\approx 0.737788$ (compare to Table 1).

Our first theorem establishes the existence of extremizers with compact support for the inequality \eqref{main ineq}, as conjectured by the second author and Ramos (see \cite{MR}, Conjecture 1.4). 

	\begin{thm}
	\label{thm:main}
		Let $w(x)$ be a symmetric decreasing weight in $L^1(\mathbb{R})$.  Then there exists a bounded, symmetrically decreasing function $\optf \in L^{1}(\R)\cap L^{2}(\R)$ (with $\optf$ depending on $w$), with compact support, such that
		\begin{equation}
			\int_{\mathbb R^2} f(x)f(y) w(x-y) dx dy = C_{opt}\|f\|_1\|f\|_2.
		\end{equation}
		
		%Moreover, the support of all extremizers has measure at most $l \le 4 \frac  {\|w\|_1^2 }{C_{opt,R}}^2.$

	%	Moreover there exists a compactly supported bounded symmetric decreasing extremizer. (I don't know if all of them are of this form, but I'd hope so)
	\end{thm}

	We start observing that 
	\begin{align*}
	C_{opt}
	&=
		\sup_{\substack{f\in L^{1}(\R)\cap L^{2}(\R) \\ 
						\|f\|_1\|f\|_2\neq0}}
		\frac{\int_{\mathbb R^2} f(x)f(y)w(x-y) dx dy }
			 {\|f\|_1\|f\|_2}
	\\&=
		\sup_{\substack{f\in L^{1}(\R)\cap L^{2}(\R) \\ 
						\|f\|_1\|f\|_2\neq0 ,f\geq0}}
		\frac{\int_{\mathbb R^2} f(x)f(y)w(x-y) dx dy }
			 {\|f\|_1\|f\|_2}
	\\&= 	
		\sup_{\substack{f\in L^{1}(\R)\cap L^{2}(\R) \\ 
						\|f\|_1\|f\|_2\neq0 ,f\geq0\\ f \ \text{symmetric decreasing}}}
		\frac{\int_{\mathbb R^2} f(x)f(y)w(x-y) dx dy }
			 {\|f\|_1\|f\|_2}\\
	\end{align*}
	
The last identity is a consequence of the well known Riesz rearangement inequality:

	\begin{lemma}[Riesz rearrangement inequality]
		Let $f:\R\to\R$ be a nonnegrative function, let $f^*$ be the symmetric decreasing rearrangement of $f$. That is:
		$$f^*(x) := \int_{a\ge 0} \chi_{[-\mu(\{f\ge a\})/2, \mu(\{f\ge a\})/2]}(x) da $$ 
		where $\mu$ is the usual Lebesgue measure. Then:

		\begin{equation}
			\int_{\mathbb R^2} f(x)f(y) w(x-y) dx dy \le 
			\int_{\mathbb R^2} f^*(x)f^*(y) w(x-y) dx dy
		\end{equation}
		with equality only if $f^*$ is a translation of $f$.% in the sense of XXXX.
	\end{lemma}

%	\begin{proof}
		%\newb{Aparece en wikipedia. Buscar mejor referencia! haha}
		%All the symmetric rearrangement proofs are the same.
	%\end{proof}
For $R>0$, let 
$$
	C_{opt,R}:=\sup_{\substack{f\in L^{1}(\R)\cap L^{2}(\R)\\ \|f\|_1\|f\|_2\neq0\\ \operatorname{supp}(f)\subset [-R,R]}}
				\frac{\int_{\mathbb R^2} f(x)f(y)w(x-y) dx dy }{\|f\|_1\|f\|_2}.
$$

A useful tool to establish Theorem \ref{thm:main} is the following local version of the theorem. 
\begin{proposition}\label{restricted extremizers existence}
There exists a function $\optf\in L^{1}(\R)\cap L^{2}(\R)$, supported in $[-R,R]$ such that
$$
C_{opt,R}:=\frac{\int_{\mathbb R^2} \optf(x)\optf(y)w(x-y) dx dy }{\|\optf\|_1\|\optf\|_2}
$$
\end{proposition}
This follows directly from an adaptation of \cite[Theorem 1.6]{MR}.

	\begin{proof}[Proof of Theorem \ref{thm:main}]
	%Restrict the problem to functions supported on $[-R,R]$ (for large $R$). Now we know (by the result of Jose \& João) that extremizers of this subclass do exist. Extremizers are still symmetric decreasing, and we will show that the extremizers do not change as $R$ gets large.\\

Let $R>0$, sufficiently large. By the Proposition \ref{restricted extremizers existence} we know that there exists a function $\optf\in L^{1}(\R)\cap L^{2}(\R)$ with support in $[-R,R]$ such that
\begin{equation}\label{opt constant R}
C_{opt,R}:=\frac{\int_{\mathbb R^2} \optf(x)\optf(y) w(x-y) dx dy }{\|\optf\|_1\|\optf\|_2}.
\end{equation}
We will show that the function $\optf$ do not depend on $R$ as soon as $R$ is large enough.

	The function $\optf$ maximizes the functional
	\begin{equation}
		\mathcal F (g) :=  \log \left[\int_{\mathbb R^2} g(x)g(y)w(x-y) dx dy \right]  -  \log \|g\|_1 - \frac 1 2 \log \|g\|_2^2
	\end{equation}
	over the set of functions $g\in L^{1}(\R)\cap L^{2}(\R)$ supported in $[-R,R]$. This leads to the following Euler-Lagrange equation inside the support of $\optf$:

	\begin{equation}\label{euler-lagrange}
		0 = \nabla_f \mathcal F (\optf) = \frac{2 \optf\ast w}{\int_{\mathbb R^2} \optf(x)\optf(y) w(x-y) dx dy}  - \frac 1 {\|\optf\|_1} - \frac{\optf}{\|\optf\|_2^2}.
	\end{equation}
	
	Let $a>0$ such that the support of $\optf$ is equal to $[-a,a]$. Integrating \eqref{euler-lagrange} from $-a$ to $a$ and rearranging, we have
	\begin{align*}
	2 \frac{\int_{[-a,a]\times \mathbb R} \optf(x) w(x-y) dxdy}
	{C_{opt,R}\|\optf\|_1\|\optf\|_2}
	=&
	\frac{1}{\|\optf\|_2^2}\int_{-a}^{a}\optf(x)dx+\frac{2a}{\|f\|_1} 
	\\=&  
	\frac{\|\optf\|_1}{\|\optf\|_2^2} + \frac{2a}{\|\optf\|_1}.
	\end{align*}

	On the other hand
	$$
	\int_{[-a,a]\times \mathbb R} \optf(x) w(y-x) dxdy = \|w\|_1\|\optf\|_1.
	$$
	Combining both equalities we obtain
	\begin{equation}
	\label{eq:support_bound_coarse}
	a = \frac{\|\optf\|_1}{\|\optf\|_2} \left (\frac{\|w\|_1}{C_{opt,R}} - \frac 1 2 \frac{\|\optf\|_1}{\|\optf\|_2} \right ) \le \frac{\|\optf\|_1 \|w\|_1 }{\|\optf\|_2C_{opt,R}},
	\end{equation}
	moreover, from H\"older's inequality we know that ${\|\optf\|_1} \le \sqrt{2a} {\|\optf\|_2}$, and therefore

	\begin{equation}
		a \le 2 \frac  {\|w\|_1^2 }{C_{opt,R}^2}. 
	\end{equation}
	%\newpage
	\textbf{Finishing the proof.}
The last remaining step is showing that the non-conmpactly-supported problem has a solution as well. That will follow if we show that
	\begin{equation}
	    C_{opt} = \lim_{R\to\infty} C_{opt,R}
	\end{equation}
	as in that case, an extremizer witnessing $C_{opt,R}$ for $R$ large enough will witness $C_{opt}$ as well. Clearly $C_{opt} \ge \lim_{R\to\infty} C_{opt,R}$ so we must show the opposite. Let $\epsilon>0$. Let $g_0\in L^1(\mathbb{R})\cap L^2(\mathbb{R})$ be an almost-extremizer, such that:
	\[
	C_{opt}\le \frac{\int_{\mathbb R^2} g_0(x)g_0(y)w(x-y) dx dy }{\|g_0\|_1\|g_0\|_2} + \epsilon/2
	\]
	now take $g_1$ compactly supported with $\|g_1-g_0\|_{1}+\|g_1-g_0\|_{2} <\delta$. If $\delta$ is small enough, that implies that
	
	\[
	\left|
	\frac{\int_{\mathbb R^2} g_0(x)g_0(y)w(x-y) dx dy }{\|g_0\|_1\|g_0\|_2}
	-
	\frac{\int_{\mathbb R^2} g_1(x)g_1(y)w(x-y) dx dy }{\|g_1\|_1\|g_1\|_2}
	\right| < \epsilon/2.
	\]
	
	Let $[-R_1,R_1]$ be the support of $g_1$. Then $C_{opt} \le C_{opt,R_1}-\epsilon$.
	\end{proof}

\begin{remark}
	%Using control of $\|w\|_2$, 
	Finer bounds for $a$ can be found by squaring and integrating \eqref{euler-lagrange}, in the form of:
	\[
		\frac{4}
			 {C_{opt,R}\|\optf\|_1^2\|\optf\|_2^2} 
		\int_{-a}^a |w*\optf|^2  dx =
				\int_{-a}^a \left |
			\frac{1}{\|\optf\|_1}+\frac{\optf}{\|\optf\|_2^2}
		\right |^2 dx,
	\]
%	Intermediate step:
%		\frac{4 \|w\|_2^2}
%			 {C_{opt,R}\|\optf\|_2^2} 
%	\ge&
%		\frac{2a}{\|\optf\|_1^2} + 3 \frac{1}{\|\optf\|_{2}^2}
and using Young's inequality in the left hand side 
\begin{align*}
	{\|w\|^2_2\|\optf\|^2_1}
		\geq  
		\int_{-a}^a |w*\optf|^2  dx, 
\end{align*}
which leads to

	\[
	\frac{\|\optf\|_1}{\|\optf\|_2} 
	\ge
	\sqrt{2a} \left(\frac{4 \|w\|_2^2}{C_{opt,R}} - 3\right)^{- \frac 1 2}.
	\]
	This inequality gives a better bound for \eqref{eq:support_bound_coarse}, giving a lower bound for the substracted term, namely the bound:
	\begin{equation}
	\label{eq:support_bound_fine}
	a \le  
	2\left (
		\frac{\|w\|_1}{C_{opt,R}}
			- 
		\sqrt{\frac{a}2} 
		\left(
			\frac{4 \|w\|_2^2}
				 {C_{opt,R}} - 3
		\right)^{- \frac 1 2} 
	\right )^2
	\end{equation}
	The right hand side of \eqref{eq:support_bound_fine} is decreasing in $C_{opt,R}$. Therefore one can find bounds for $a$ by substituting $C_{opt,R}$ even when its exact value is not known.
\end{remark}

%% file: approximation.tex
%!TEX root = main.tex
\textbf{Notation:} Trough this section $\bra f$ will denote a row vector, $\ket f$ a column vector. If $a$ is a matrix, $\bra f a \ket g$ will be the vector-matrix-vector product. Moreover, $\ket 1$ will be the column vector with all ones. We will use the notation $\braket fg$ as a dot product for $f,g$, with the dot product $\braket fg = \delta_x \sum f_i g_i$ ($\delta_x$ is the discretization scale). If the discretized functions are supported in $(-a,a)$, then $\braket 11 = 2a$. We denote by $K_w$ the convolution operator associated to $w$.

Let 
$
L^{1,2}([a,b]):= L^{1}([a,b])\cap L^{2}([a,b]),$ and, for each function $f\in L^{1,2}([a,b])$ we define
$\|f\|_{L^{1:2}([a,b])} := \|f\|^{1/2}_{L^1([a,b])}\|f\|^{1/2}_{L^2([a,b])}$.

\begin{definition}
We define some spaces as the space of measurable functions (up to a.e. equivalence) with the following norms (with $\lambda$ a parameter greater than zero):

\begin{itemize}
    %\item %$\|f\|^2_{L^{1:2}([a,b])} = \|f\|_{L^1([a,b])}\|f\|_{L^2([a,b])}$
    \item $\|f\|_{H_\lambda([a,b])}^{2} = \lambda\|f\|_2^2+\lambda^{-1}\left|\int_a^b f dx \right|^2$.
    \item $\|f\|_{B_\lambda([a,b])}^{2} = \lambda\|f\|_2^2+\lambda^{-1}\left|\int_a^b |f| dx \right|^2$.
\end{itemize}
These spaces satisfy nice properties, including the following:
\begin{itemize}
    \item $\|f\|^2_{L^{1:2}([a,b])} = \frac 1 2 \inf_{\lambda} \|f\|^2_{B_{\lambda}([a,b])}$. This is a consequence of AM-GM inequality.
    \item $ \|f\|_{B_{\lambda}([a,b])} \ge \|f\|_{H_{\lambda}([a,b])}$ with equality if and only if $f$ is nonnegative.
    \item $H_{\lambda}$ is a Hilbert space. Since, by definition, $\|\cdot\|_{H_{\lambda}([a,b])}$ satisfies the parallelogram law.
\end{itemize}
\end{definition}
{\remark{
Observe that by H\"older's inequality and AM-GM inequality we have
\begin{align*}
&\|f+g\|^2_{H_{\lambda}([a,b])}-\|f\|^2_{H_{\lambda}([a,b])}-\|g\|^2_{H_{\lambda}([a,b])}-2\|f\|_{H_{\lambda}([a,b])}\|g\|_{H_{\lambda}([a,b])}\\
&=\lambda\|f+g\|_{L^2([a,b])}^2+\lambda^{-1}\left|\int_a^b f+g dx \right|^2\\
&\ \ \ -\lambda\|f\|_{L^2([a,b])}^2-\lambda^{-1}\left|\int_a^b f dx \right|^2-\lambda\|g\|_{L^2([a,b])}^2-\lambda^{-1}\left|\int_a^b g dx \right|^2\\
&\ \ \ -2\|f\|_{H_{\lambda}([a,b])}\|g\|_{H_{\lambda}([a,b])}\\
&\leq 2\lambda\int_{a}^{b}fgdx+\frac{2}{\lambda}\left|\int_{a}^{b}fdx\right| \left|\int_{a}^{b}gdx\right|-2\|f\|_{H_{\lambda}([a,b])}\|g\|_{H_{\lambda}([a,b])}\\
&\leq 2\lambda\|f\|_{L^{2}{([a,b])}}\|g\|_{L^{2}([a,b])}+\frac{2}{\lambda}\left|\int_{a}^{b}fdx\right| \left|\int_{a}^{b}gdx\right|-2\|f\|_{H_{\lambda}([a,b])}\|g\|_{H_{\lambda}([a,b])}\\
&\leq 0, \ \text{for any two functions}\ f,g\in H_{\lambda}([a,b]).
\end{align*}
This means that
$\|\cdot\|_{H_{\lambda}([a,b])}$ satisfies the triangle inequality, so this is really a norm (since all the other properties trivially hold).
Similarly we can see that
$\|\cdot\|_{B_{\lambda}}$} 
is in fact a norm.}

The strategy to find the optimal value $C_{opt}$ will be in two steps:

\begin{itemize}
    \item First, by the addition of a parameter $\lambda$ we will turn the optimization problem into a computtionally tractable optimization problem.
    \item We will then show that a suitably discretized version of the computationally tractable problem is quantitatively close to the original continuous problem.
\end{itemize}

\subsection*{A computationally tractable relaxation of the problem.} For any $\lambda>0$, let 
\begin{equation}\label{c lambda definition}
c_{\lambda}:=c_{\lambda}(w):=\max_{f \in L^1 \cap L^2}  2 \frac{\bra f w \ket f}{\lambda  \braket f f + \lambda^ {-1}\braket f 1 \braket 1 f}.
\end{equation}
Observe that by Fubini's theorem and H\"older's inequality we have
\begin{equation}\label{upper bounds for c lambda depending on lambda}
    c_{\lambda}\leq \min\{2\lambda,2/\lambda\}.
\end{equation}
Using the fact that $\min_{\lambda >0} \lambda a^2+\lambda^ {-1} b^2 = 2 ab$ by AM-GM inequality. We can turn our original problem into:
$$
C_{opt}:=c(w) :=\max_{\lambda>0}c_{\lambda}(w)= \max_{\lambda>0} \max_{f \in L^1 \cap L^2}  2 \frac{\bra f w \ket f}{\lambda  \braket f f + \lambda^ {-1}\braket f 1 \braket 1 f}. 
$$

\begin{lemma}\label{relation between norms in H lambda and L12}

We have that 
$$
\max_{\|f\|_{L^{1:2}([a,b])}\le 1} \langle f|K_w|f\rangle = 2\max_{\lambda>0} \max_{\substack{\|f\|_{H_\lambda ([a,b])}\le 1\\f\ge 0
             }
     }
\langle f|K_w|f\rangle
$$
moreover, the extremizers to $\max_{\|f\|_{H_\lambda ([a,b])}\le 1} \langle f|K_w|f\rangle$ (which exist by the Hilbert theory) are symmetric decreasing nonnegative.
\end{lemma}
\begin{proof}
Observe that
\begin{align*}
    \max_{\|f\|_{L^{1:2}([a,b])}\le 1} \langle f|K_w|f\rangle &=\max_{f\in L^{1:2}([a,b])} \frac{\langle f|k_w|f\rangle}{\|f\|^2_{L^{1:2}[(a,b)]}}  \\
    &= 2\max_{f\in L^{1:2}([a,b])} \frac{\langle f|k_w|f\rangle}{\inf_{\lambda>0}\|f\|^2_{B_{\lambda}[(a,b)]}}\\
    &= 2\max_{\substack{f\in {H_\lambda ([a,b])}\\f\ge 0}} \frac{\langle f|k_w|f\rangle}{\inf_{\lambda>0}\|f\|^2_{H_{\lambda}[(a,b)]}}\\
    &=2\max_{\lambda>0} \max_{\substack{\|f\|_{H_\lambda ([a,b])}\le 1\\f\ge 0}}
\langle f|K_w|f\rangle.
\end{align*}
The last part of the statement follows from the Riesz rearrangement inequality.
\end{proof}

The problem of finding $c_{\lambda}$ for a fixed $\lambda$ becomes now essentially a problem about finding the spectrum of $w$ (as a convolution operator) in a certain Hilbert space, with certain subtleties arising from the fact that we have a restriction to $f\ge 0$. These subtleties are addressed in Section 3.3. 

\subsection*{Relating a discretized version of the problem to the continuous problem.}  We start by defining our discretized spaces as the space of step functions on intervals of length $\delta$:

\begin{definition}
 Given $\delta>0$ the set $V_\delta$ will be the set of functions that are constant on intervals of the form $[n\delta, (n+1)\delta)$. Given a function $f$ we define $[f]_\delta \in V_\delta$ by $[f]_\delta = \delta^{-1}\int_{\delta n}^{\delta(n+1)} f(s) ds$. We also define $\{f\}_{\delta} := f-[f]_{\delta}$.
\end{definition}

Observe that $\|[f]_{\delta}\|_1=\|f\|_1$ and by H\"older's inequality $\|[f]_{\delta}\|_2\leq \|f\|_2$ for all $f\in L^{1}(\R)\cap L^{2}(\R)$. 
Moreover, by the previous observation, \begin{equation}\label{relacion entre normas en H}
\|[f]_{\delta}\|_{H_{\lambda}}\leq \|f\|_{H_{\lambda}} \end{equation}
for all $f\in H_{\lambda}$.

Our next lemma establishes smoothness properties for the extremizers on the interior of the support. %under mild conditions on the weight $w$.
Regularity (in the form of \ref{l-infinity bound for the derivative}) will be crucial  %in the following sections, 
in order to implement a numerical scheme to find the constants $C_{opt}$.

\begin{lemma}\label{l-infinity bound for the derivative}
Let $\hat f$ be an extremizer to the problem $\max_{\|f\|_{H_\lambda ([a,b])}\le 1} 2\langle f|K_w|f\rangle
$.\\
Let $c_\lambda =  \max_{f\ge 0; \|f\|_{H_\lambda ([a,b])}\le 1} 2 \langle f|K_w|f\rangle
$. Then  $\|\hat f\|_{2} \le \lambda^{-1/2}$ and $\|\hat f'\|_{2} \le \frac{4}{c_\lambda \lambda^2}$.
\end{lemma}

%\textbf{Creo que lo que obtenemos es:}
%$$
%\frac{8}{\lambda^{5/2}c^2_{\lambda}}?
%$$

\begin{proof}
    The fact that $\|\hat f\|_{2} \le \lambda^{-1/2}$ follows from the fact that $\|\hat f\|_{2} \le \lambda^{-1/2}\|\hat f\|_{H_{\lambda}}$.

For the second part, observe that the function $\hat f$ maximizes the functional
	\begin{equation*}
		\mathcal F (g) :=  \log \langle g|k_w|g\rangle   -  \log \|g\|^2_{H_{\lambda}([a,b])}
	\end{equation*}
	over the set of functions  $g\in H_{\lambda}([a,b])$. This leads to the following Euler-Lagrange equation inside of the support of $\hat f$:

	\begin{equation}\label{euler-lagrange2}
		0 = \nabla_f \mathcal F (\hat f) = \frac{2 k_w \hat f}{\langle \hat f|k_w|\hat f\rangle}  - \frac {2\lambda \hat f+2\lambda^{-1}I \hat f} {\|\hat f\|_{H_{\lambda}([a,b])}} ,
	\end{equation}
	where $I$ is the operator that maps $f$ to a function with value equal to the integral of $f$.
	Then, multiplying by $\|\hat f\|_{H_{\lambda}([a,b])}$, and using the definition of $c_{\lambda}$, we obtain
		\begin{equation*}
		0 = \nabla_f \mathcal F (\hat f) = \frac{4 k_w\hat f}{c_{\lambda}}  -  {2\lambda \hat f-2\lambda^{-1}I\hat f}.
	\end{equation*}
		Therefore, since $\hat f\geq0$ we have that
	\begin{equation}\label{euler lagrange eq for lambda problem}
	\hat f=\max\{0,[c^{-1}_{\lambda}(\lambda+\lambda^{-1}I)^{-1}2k_w]\hat f\}.
	\end{equation}
    This implies that on the support of $f$, it holds that 
$$|\hat f ' |\leq  \frac{2}{c_\lambda \lambda} |(w \ast \hat f)'|$$
then, by Young's convolution inequality
$$
    \|\hat f '\|_2\le \frac{2}{c_\lambda \lambda} \|w'\|_{1}\|\hat f\|_2=\frac{2}{c_\lambda \lambda} \|w\|_{TV} \|\hat f\|_2  = \frac{4}{c_\lambda \lambda} \|w\|_{\infty} \|\hat f\|_2 \le \frac{4}{c_\lambda \lambda^{3/2}}.
    $$
\end{proof}

We are now ready to state the main result of the section:

\begin{proposition}\label{bounding the error}
    Let $c_{\lambda,\delta}:= \sup_{f\ge 0, f\in V_\delta,\|f\|_{H_{\lambda}}\le 1} \langle f |K|f\rangle$ be the discretized version of $c_{\lambda}$. Then:

    \begin{align}
0 \le c_\lambda-c_{\lambda,\delta} \le& \frac{16\delta^2}{\pi^2c_{\lambda}\lambda^2}
\end{align}

\begin{proof}
By construction, $c_\lambda\ge c_{\lambda,\delta}$, so we can focus on the second inequality. Let $f^*$ be an extremizer for $c_\lambda$, we assume without loss of generality that $\|f^*\|_{H_{\lambda}}=1$.
Then, by \eqref{relacion entre normas  en H} we have $\|[f^*]_{\delta}\|_{H_{\lambda}}\leq 1$, moreover
$$
c_{\lambda}-c_{\lambda,\delta}\leq K_w(f^*,f^*) - \frac{K_w([f^*]_\delta, [f^*]_\delta)}{\|[f^*]_{\delta}\|^2_{H_{\lambda}}}\leq K_w(f^*,f^*) - K_w([f^*]_\delta, [f^*]_\delta).
$$
Thus, it suffices to bound 
$$
K_w(f^*,f^*) - K_w([f^*]_\delta, [f^*]_\delta) = K_w([f^*]_\delta + f^*, \{f^*\}_\delta)  = (([f^*]_\delta + f^*)\ast w, \{f^*\}_{\delta}).
$$
In order to do this, we use two key properties: 
\begin{itemize}
    %\item (Free $\delta$ improving) The map $f\mapsto[f]_{\delta}$ (and therefore $\{\cdot\}_{\delta}$ )is an $L^2$ (and an $H_\lambda$) orthogonal projection. In particular $(g,\{f\}_{\delta}) = (\{g\}_{\delta}, \{f\}_{\delta})$.
    \item (Orthogonality) $([f]_{\delta},\{g\}_{\delta})=0$ for any two functions $f,g\in L^{1}(\R)$.
    \item (Optimal Poincare's inequality) $\|\{f\}_{\delta}\|_{L^2}\le \frac \delta {\pi} \|f'\|_{L^2} $.
\end{itemize}

%By the first estimate,
Using the orthogonality property and Young's convolution inequality we obtain
$$
c_\lambda-c_{\lambda,\delta} \le \|\{(f^* +[f^*]_\delta)\ast w\}_\delta\|_2 \|\{f^*\}_{\delta}\|_2.
$$
Moreover, by the optimal Poincare's inequality and Young's convolution inequality
$$
\|\{(f^* +[f^*]_\delta)\ast w\}_\delta\|_2\leq 
\frac{\delta}{\pi} \|(f^* +[f^*]_\delta)*w'\|_2
\leq \frac{\delta}{\pi}\|(f^* +[f^*]_\delta\|_2\|w\|_{TV}.
$$
Therefore, once again, by the optimal Poincaré inequality %(together with the fact that $\|(f^*\ast w)'\|_2 \le \|f^*\|_2\|w\|_{TV}$),
$$
c_\lambda-c_{\lambda,\delta} \le \frac{\delta^2}{\pi^2}\|f^* +[f^*]_\delta\|_2\| w\|_{TV} \|{f^*}'\|_2
$$

The proposition now follows using that $\|[f^*]_{\delta}\|_2\leq \|f^*\|_{2} \le \lambda^{-1/2}$, $\|\hat {f^*}'\|_{2} \le \frac{4}{c_\lambda \lambda^{3/2}}$, $\|w\|_{TV}=2\|w\|_{\infty}=2$, and the triangle inequality, to obtain:

\begin{equation}
    c_\lambda-c_{\lambda,\delta} \le \frac{16\delta^2}{\pi^2c_{\lambda}\lambda^2}.
\end{equation}

\end{proof}

\end{proposition}

%% file: computational.tex
%!TEX root = main.tex
%Observe that 

\subsection{Computational aspects I: Regularity with respect to \texorpdfstring{$\lambda$}{lambda}} The computational strategy will be to find the value $c_\lambda$ for a discrete subset of the possible $\lambda$, and prove regularity properties of $c_\lambda$ that guarantee that $C_{opt}= c_{\lambda^*}$ is not far from the maximum of the values of $c_{\lambda}$.

\begin{lemma}
The following properties hold:
\begin{itemize}
    \item[(i)] The Lipstchitz constant of $c_{\lambda}$ is at most $1$.
    \item[(ii)] If $C_{opt}=c_{\lambda^*}$, then
    $$
    c_{\lambda} \ge \frac{2 c_{\lambda^*} }{\lambda^{-1}\lambda^*+\lambda{\lambda^*}^{-1}}
    $$
\end{itemize}
\end{lemma}
\begin{proof}
%\newb{Clarify the proof of the first part of the lemma}
{\it{Proof of (i):}} Let $g$ be a function such that $\|g\|_1 = 1$, then $\langle g|K_w|g\rangle\le 1$, and 
$$
\left|\frac d {d\lambda} (\lambda \|g\|^2_2+ \lambda^{-1} \|g\|^2_1)^{-1}\right| = \lambda^ {-1}(\lambda \|g\|^2_2+ \lambda^{-1} \|g\|^2_1)^{-2} |\lambda \|g\|^2_2- \lambda^{-1}\|g\|^2_1|\le 1.
$$
The result follows from this.\\

{\it{Proof of (ii)}:} %For the second, 
We start observing that
$$
\lambda^{*}=\frac{\|f^*\|_1}{\|f^*\|_2},
$$
where $f^*$ is an extremizer for our original problem (This is when the equality happen in AM-GM inequality).
Then
\begin{align*}
 c_{\lambda}({\lambda^{*}}^{-1}\|f^*\|^2_1+\lambda^{*}\|f^*\|^2_{2})\geq 2\langle f^*|K_w|f^*\rangle= 2\|f^*\|_1\|f^*\|_2c_{\lambda^*}.
\end{align*}
Therefore
\begin{align*}
    c_{\lambda}\geq \frac{2\|f^*\|_1\|f^*\|_2c_{\lambda^*}}{{\lambda^{*}}^{-1}\|f^*\|^2_1+\lambda^{*}\|f^*\|^2_{2}}=\frac{2c_{\lambda^*} }{\lambda^{-1}\lambda^*+\lambda{\lambda^*}^{-1}}.
\end{align*}
\end{proof}

%\textbf{------------------}
\subsection{Computational aspects II: Discretizing the convolution kernel.}
Let $a,b\in \delta\mathbb{Z}$, and
as previously, let $V_\delta$ be the set of functions $g:[a,b)\to\mathbb{R}$ that are constants on intervals of length $\delta$.
Let $f:\delta \mathbb Z\cap [a,b)\to\mathbb{R}$, we define $\|f\|_{l^p_\delta}:= \left (\sum_{i\in\delta\mathbb Z\cap[a,b)} f(i)^p \delta\right )^{1/p}$ (essentially a $\delta-$discretization of the $L^p$ norm), and, we define $E[f]:=\sum_{i\in\delta\mathbb Z} f(i)\chi_{[i,i+\delta]}\in V_{\delta}$. Observe that we have the equality $\|f\|_{l^p_\delta} = \|E[f]\|_{L^p}$. For any $a,b\in \delta \mathbb Z$ this defines a natural isomorphism between $V_{\delta}$ and functions with domain $\delta \mathbb Z \cap [a,b)$. Let
$$
\iota :l^2(\delta \mathbb Z \cap [a,b)) \to L^2(\mathbb R \cap [a,b))\cap V_{\delta}
$$ be the map given by this isomorphism i.e $if=E[f]$ for all $f\in \delta\mathbb{Z}\cap[a,b)$ 
(Here we are considering $\delta \mathbb Z$ with $\delta$ times the counting measure as a measure). This isomorphism allows us to define the metrics induced in $V_{\delta}$ by  %$L^{1:2}$, 
$H_\lambda$, $B_\lambda$ as a metric for functions in $\delta \mathbb Z \cap [a,b)$. %Let $$\iota :l^2(\delta \mathbb Z \cap [a,b)) \to L^2(\mathbb R \cap [a,b))\cap V_{\delta}$$ be the map given by this isomorphism.

Given a function $w \in L^1 \cap L^{\infty} (
[a,b])$ there is a function in $\tilde w \in l^1 \cap l^{\infty} (\delta \mathbb Z \cap [a,b])$ such that for $f,g \in l^2(\delta \mathbb Z \cap [a,b))$ it holds that $$\langle  \iota f, w\ast_{\mathbb R} \iota g \rangle_{L^2(\mathbb R)} =  \langle  f, \tilde w\ast_{\delta \mathbb Z}  g \rangle_{l^2(\delta \mathbb Z)}. $$

The function $\tilde w$ is given explicitly by:

\begin{align*}
\tilde w(s) 
=&
\langle \delta^{-1} 1_s, \tilde w\ast_{\delta \mathbb Z} (\delta^{-1} 1_{0}) \rangle_{l^2(\delta \mathbb Z)} 
\\=&
\langle \delta^{-1} \chi_{[s,s+\delta)},  w\ast_{ \mathbb R} (\delta^{-1} \chi_{[0,\delta)}) \rangle_{L^2(\mathbb R)} 
%\\=&
%\langle \delta^{-1} 1_s, \tilde w\ast_{\delta \mathbb Z} (\delta^{-1} 1_{0}) \rangle_{l^2(\delta \mathbb Z)} 
\\=&
\fint_{s}^{s+\delta}\fint_{0}^{\delta} w(y-x)  dx dy
\\=&
\delta^{-2} \int_{s}^{s+\delta} w(t) (\delta-|t-s|) dt.
\end{align*}
In the two cases of special interest (when $w$ is a Gaussian function or when $w$ is the characteristic function of a set, we get more explicit values). In practice, however, the high stability of the integrals make it more accurate to perform the integrals numerically if $\delta$ is small than to compute the difference numerically (of the order of $10^{-5}$), for more details see the annotated code.

\subsection{Computational aspects II: Solving the problem for a fixed \texorpdfstring{$\lambda$}{lambda} at a discretization scale \texorpdfstring{$\delta$}{delta} }

What remains to do now is to give an algorithm that allows us to solve the extended problem for a fixed value of $\lambda$. The key fact that we use is that, if we %were able to 
remove the constraint $f\ge 0$, then the solution could be readily found using the power method for self-adjoint finite dimensional linear operators (symmetric matrices). We focus first in the situation when the constraint $f\ge 0$ is removed, and then argue that we can assume we are in that situation. \\

Recalling that
$$
C_{opt}(w) = \max_{\lambda>0} \max_{f \in L^1 \cap L^2([-a,a])}  2 \frac{\bra f w \ket f}{\lambda  \braket f f + \lambda^ {-1}\braket f 1 \braket 1 f}. 
$$
Assume that the functions under consideration are supported in $[-a,a)$, so we have $\braket 11 = 2a$. In an abuse of notation, we denote by $w$ the convolution operator associated to $w$.
The key observation is the following
$$
{\lambda \braket f f + \lambda^{-1} \braket f 1 \braket 1 f} = \bra f  (\sqrt \lambda Id + b_{\lambda} \ket 1 \bra 1 )^2 \ket f = \bra f A_{\lambda}^2 \ket f,
$$
where $b_{\lambda}$ is the unique positive solution to $\lambda^{-1} = 2\sqrt{\lambda} b_{\lambda}+2a b_{\lambda}^2$ and $A_{\lambda}$ is a hermitian positive definite matrix. Let $g:= A_{\lambda}f$, then we have:

$$
2 \frac{\bra f w \ket f}{\lambda  \braket f f + \lambda^ {-1}\braket f 1 \braket 1 f}  = 2 \frac{\bra g A_{\lambda}^{-1}  w A_{\lambda}^{-1}  \ket g}{\braket gg}
$$
defining $M_{\lambda}:= 2 A_{\lambda}^{-1}  w A_{\lambda}^{-1} $ we have that

$$
C_{opt}(w) = \max_{\lambda>0} \max \operatorname{Spec}(M_{\lambda}).
$$

%\subsection{Solving the auxiliary problem}

What remains to be shown is that the constraint $f\ge 0$ can indeed be dropped out. Let $f^*$ be the solution to the discretized problem constrained to $f\ge 0$. We can assume (by the discrete version of Riesz rearrangement inequality, see \cite[Chapter X]{HLP}) that $f^*$ is symmetric and non-increasing. Assume that, out of all potential extremizers to the discretized problem, $f^*$ has minimal support. \\

Let $\tilde V$ be the (finite dimensional) space of functions in $\delta \mathbb Z$ with the same support as $f^*$. Since $f^*$ extremizes the rayleigh quotient $\langle f|w|f\rangle$  it must be an eigenvector $w$ restricted to the support of $f$. It is the unique eigenvector: If there was another eigenvector $g^*$, the function $g^*-\alpha f^*$ would be an eigenvector as well, with a strictly smaller support if $\alpha$ is chosen appropriately.\\

Let $l$ be the number of $\delta$-intervals in the discretized shortest support. We can further constrain the optimization problem to functions supported in these $l$ intervals without changing the value of the optimization. In this case, the extremizer function $f^*$ will be an extremizer on the interior, and therefore, the largest eigenvalue of the bilinear form induced by $w$. We can find such eigenvalue by the power method.\\

This shows the following algorithm will give the extreme value up to discretization erros:

\begin{enumerate}
	\item Choose $\delta$ small enough for the desired error in Proposition \ref{bounding the error} to hold. Let $N$ be the number of $\delta-$intervals intersecting $[-a,a]$
	\item Fix $\lambda$ (and then repeat for a large enough set of $\lambda$ for the desired tolerance in Proposition \ref{bounding the error} to hold).
    \item For each natural $1\le k\le N$ solve the unconstrained discretized problem (the eigenvalue problem) with the power method. If the solution satisfies the constraints, keep as a potential problem.
    \item The maximum for our original problem (up to the tolerance/error arising from the discretization) is the maximum over all the solutions given in the previous step, where the maximum is taken over all the $\lambda$ and $k$.
\end{enumerate}

\subsection{A conjectured iterative method}

The Euler-Lagrange equations for $f$ in the support of $f$ can be written as:

	\begin{equation}\label{euler-lagrange-2}
		 \frac \optf {\|\optf\|_2^2} =  \max\left (\frac{2 \optf\ast w}{\int_{\mathbb R^2} \optf(x)\optf(y) w(x-y) dx dy}  - \frac 1 {\|\optf\|_1},0\right ).
	\end{equation}
This suggests a fixed point method (both in the discrete and continuous set-ups) to find an extremizer to the autoconvolution problem. This method converges in practice (see the last column in Table \ref{table}), and is hundreds of times faster than the method described in the previous section. We have however been unable to show convergence of such method. 

This method bears resemblance to the numerical method in \cite{CDGJLM}, where the authors were not able to show convergence of the method either.

\section{A comparison with the maximum problem}

At first glance this problem bears resemblance to the original maximum problem originally studied by Cilleruelo, Rusza and Vinuesa \cite{CRV}. The problem of the mean (when $w = \chi_{[-1/2,1/2]}$) was indeed posed by Steinberger as a relaxation of the problem of the maximum, namely finding the largest constant such that:

\begin{equation}
    \label{eq:maximo}
\text{max}_{-1/2 \le t \le 1/2} \int_{\R} f(t-x)f(x) \, dx \ge c_{\max} \left(\int_{-1/4}^{1/4} f(x) \, dx\right)^2.
\end{equation}

this problem was studied using computational methods by Cloninger and Steinberger \cite{CS}. The complexity of numerical method proposed in \cite{CS} grows exponentially in the level of discretization of the function, as opposed to the polynomial cost in the method proposed in this work.

One may hope that the methods in this work could generalize to methods in the maximum problem, however we believe that the problems are genuinely different for two reasons:

\begin{itemize}
    \item There is no reason to expect smooth (or Liptschitz) global extremizers to the maximum problem. In fact, the best known numerical extremizers \cite[Figure 1]{MV} seem non-smooth, and are certainly not symetrically decreasing. This prevents the gain of a $\delta^2$ in the discretization of the problem, which now has to be done much more carefully (\cite[Lemma 1]{CS}). 
    \item The (discretized) problem of the maximum seems to have multiple local extremizers of the functional in equation \eqref{eq:maximo}. This quite likely prevents fixed point methods to be efficient.  
\end{itemize}

%% file: another_problem.tex
For any $f\in L^{1}(\R)\cap L^2(\R)$ the following inequality hods: 
\begin{equation*}
     \|f\ast f\|_{2} \le 
     \|f\ast f\|_{1}
     \|f\ast f\|_{\infty}.
\end{equation*}
Motivated by the regularizing effect of the autoconvolutions it was conjectured by Martin and O’Bryant [\cite{MOB3}, Conjecture 5.2.] that there is a
universal constant $0<c<1$ such that for all nonnegative $f\in L^{1}(\R)\cap L^2(\R)$ the following improved inequality hods: 
\begin{equation*}\label{ineq thm 2}
     \|f\ast f\|_{2} \le c
     \|f\ast f\|_{1}
     \|f\ast f\|_{\infty}.
\end{equation*}
In fact, they proposed $c=\frac{\log 16}{\pi}\sim 1-0.1174$. This was disproved by Matolcsi and Vinuesa \cite{MV}, they observed that a necessary condition is $c\geq 1-0.1107$.

In this section we show that 

\begin{theorem}\label{otro teorema de autoconvolutions}
The best constant $c$ such that for any $f\in L^{1}(\R)\cap L^2(\R)$ the inequality 
\begin{equation}
    \|f\ast f\|_{2} \le c
     \|f\ast f\|_{1}
     \|f\ast f\|_{\infty}
\end{equation}
holds is $1$.
\end{theorem}

We start with two results about cut-off functions in $\mathbb R$.

\begin{prop}
There exists a smooth real function $\eta \in L^1(\mathbb R) \cap L^\infty(\mathbb R)$ and some $n>100$ such that the following hold:

\begin{itemize}
	\item $\eta$ is symmetric nonnegative
	\item $\hat \eta$ is $\mathcal C^{n-1}$, symmetric nonnegative, supported on $[-1,1]$ and positive on $(-1,1)$
	\item $\|\eta\|_1=1$
	\item $\hat\eta(1-x) = c x^n(1+o(1))$, for $x \in [0,\frac 1 {10}]$, where the $o(1)$ term is smooth.
\end{itemize}
\end{prop}

\begin{proof}
	Let $f = (\frac 1 4-x^2)^m \chi_{[-1/2,1/2]}$ for $m>100$. Let $\hat \eta = c f\ast f$, for $c$ that makes  $\|\eta\|_1=1$ hold.
\end{proof}

\begin{prop}
\label{can-take-sqrt}With $\eta$ as given in the previous lemma, if $\phi:\mathbb R\to \mathbb R$ is a smooth function such that $\phi(x) = \bar \phi(-x)$ for all $x\in\mathbb{R}$, $\phi(1)=1+o(1)$, and it does not vanish on $[-1,1]$, then $\phi \hat \eta$ admits a $\mathcal C^{50}$ (complex) square root on $\mathbb R$.
\end{prop}

\begin{proof}
By elementary complex analysis, we know there exists a unique $C^{0}_c$ complex square root  $f$ of $\phi \hat \eta $. Moreover it is smooth in $(-1,1)\cup [-1,1]^c$. It suffices to show that this square root is $\mathcal C^{50}$ near $1$. By hypothesis, $\phi = (1+o(1))$ near $1$, where the $o(1)$ is a smooth term. Therefore $\phi \hat \eta (1-x)$ is $cx^{n}(1+o(1))$, where the $o(1)$ is a smooth term, and therefore admits a $\mathcal C^{50}$ square root.
\end{proof}

\begin{proof}[Proof of Theorem \ref{otro teorema de autoconvolutions}]
Throughout the proof, let $\eta$ as given in the theorem above. Let $\eta_c(x) = c^{-1}\eta(c x)$. Let {$g = \chi_{[-1/2,1/2]}\ast \eta_c$} for $c$ large enough. 

We have that $|g|$ is a real-valued Schwartz function, and that 

\begin{equation}
    \|g\|_{2} \ge (1-\epsilon)
     \|g\|_{1}
     \|g\|_{\infty}
\end{equation}
for $\epsilon=\epsilon(c)$ as small as we want. We will show that there exists a test function $f$ such that the following inequality holds
$$
\|f\ast f- g\|_*<\epsilon
$$
for $*=1,2,\infty$. This will finish the proof. Note that we can  get $*=2$ by interpolation. Let $z_1,\dots z_n$ be the positive zeros of $\hat g$. Note that sincce $\hat g(-x) = \overline{\hat g}(x)$, the negative zeros are $-z_1,\dots -z_n$. Let $\delta$ sufficiently small, %to be chosen later, 
and let 
$$
\hat {\tilde g}(w):= \hat \eta_c \left(\hat \chi_{[-1/2,1/2]} +i\delta  \sum_{j=1}^n  \hat \eta(\delta^{-1}(w-z_j))-\hat \eta(\delta^{-1}(w+z_j))\right).
$$
Now we can take a square-root of $\hat {\tilde g}(w)$ by Proposition \ref{can-take-sqrt} and the result follows.
\end{proof}

%% file: main.bbl
\begin{thebibliography}{99}

\bibitem{BarnardSteinerberger} 
\newblock R.C. Barnard, S. Steinerberger,
\newblock \emph{Three convolution inequalities on the real line with connection to additive combinatorics.}
\newblock Journal of Number Theory, \textbf{207} (2020), 42--55. 

\bibitem{Beckner}
\newblock W. Beckner, 
\newblock \emph{Inequalities in Fourier analysis.} 
\newblock  Annals of Mathematics, \textbf{102} (1975), 159--182. 

\bibitem{BCK}
\newblock  J. Bourgain, L. Clozel and J.P. Kahane, 
\newblock \emph{Principe d’Heisenberg et fonctions positives.}
\newblock Annales de l’institut Fourier, \textbf{60} (2010), n. 4, 1215--1232.

\bibitem{CDGJLM}
 A. Chang, J. de Dios, R. Greenfeld, A. Jamneshan, Z. Li and J. Madrid,
 \newblock \emph{Decoupling for fractal subsets of the parabola.}
\newblock Preprint https://arxiv.org/abs/2012.11458

\bibitem{CJLL}
 E. Carlen, I. Jauslin, E. Lieb and M. Loss,
 \newblock \emph{On the Convolution Inequality $f\geq f*f$.}
\newblock To appear in IMRN https://doi.org/10.1093/imrn/rnaa350.

\bibitem{CRT}
\newblock J. Cilleruelo, I. Ruzsa and C. Trujillo,
\newblock \emph{Upper and lower bounds for finite $Bh[g]$ sequences.}
\newblock Journal of Number Theory, \textbf{97} (2002), n. 1, 26--34.

\bibitem{CRV} 
\newblock J. Cilleruelo, I. Ruzsa and C. Vinuesa,
\newblock \emph{Generalized Sidon sets.}
\newblock Advances in Mathematics, \textbf{225} (2010), n. 5, 2786 --2807.

\bibitem{CS}
\newblock A. Cloninger and S. Steinerberger, 
\newblock \emph{On suprema of autoconvolutions with an application to Sidon sets.}
\newblock Proceedings of the American Mathematical Society \textbf{145} (2017), no. 8, 3191--3200.

\bibitem{CG}
\newblock H. Cohn and F. Gon\c{c}alves,
\newblock \emph{An optimal uncertainty principle in twelve dimensions via modular forms.}
\newblock Inventionnes Mathematicae {\bf 217} (2019), no.~3, 799--831. 

\bibitem{Github}
\newblock  J. de Dios, J. Madrid, 
\newblock \emph{On classical inequalities for autocorrelations and autoconvolutions: accompaining code.}
\newblock Hosted on GitHub at \url{https://github.com/jaumededios/suprema-autocorrelations}.

\bibitem{FKM}
\newblock S. Fish, D. King and S. J. Miller,  
\newblock \emph{Extensions of Autocorrelation Inequalities with Applications to Additive Combinatorics.}
\newblock arXiv preprint \texttt{arXiv:2001.02326}. 

\bibitem{Folland}
\newblock G. Folland, 
\newblock \emph{A course in abstract harmonic analysis.} 
\newblock Chapman and Hall/CRC, 2016. 

%\bibitem{GOeSR2}
%\newblock F. Gon\c{c}alves, D. Oliveira e Silva, and J. P. G. Ramos, 
%\newblock \emph{New sign uncertainty principles.}
%\newblock Preprint, 2020.

%\bibitem{GOeSR1}
%\newblock F. Gon\c{c}alves, D. Oliveira e Silva, and J. P. G. Ramos, 
%\newblock \emph{On regularity and mass concentration phenomena for the sign uncertainty principle.}
%\newblock Preprint, 2020.

\bibitem{GOeSS}
\newblock F. Gon\c{c}alves, D. Oliveira e Silva, and S. Steinerberger,
\newblock \emph{ Hermite polynomials, linear flows on the torus, and an uncertainty principle for roots.} 
\newblock Journal of Mathematical Analysis and Applications {\bf 451} (2017), no.~2, 678--711. 

\bibitem{G}
\newblock B. Green, 
\newblock \emph{The number of squares and $Bh[g]$ sets.}
\newblock Acta Arithimetica, \textbf{100} (2001), 365--390.

\bibitem{HLP}
\newblock G.H. Hardy, J.E. Littlewood, and G. Polya,
\newblock \emph{Inequalities}
\newblock Cambridge University Press, Cambridge (1934).

\bibitem{Ha}
\newblock J. K. Haugland, 
\newblock \emph{Advances in the minimum overlap problem.}
\newblock Journal of Number Theory, \textbf{58} (2001), 71--78.

\bibitem{MOB1}
\newblock G. Martin and K. O'Bryant, 
\newblock \emph{Constructions of generalized Sidon sets.}
\newblock Journal of Combinatorial Theory Series A \textbf{113} (2006), n. 4, 591--607.

\bibitem{MOB2} 
\newblock G. Martin and K. O'Bryant,
\newblock \emph{The symmetric subset problem in continuous Ramsey theory.}
\newblock Experimental Mathematics \textbf{16} (2007), n. 2, 145--166.

\bibitem{MOB3}
\newblock G. Martin, K. O’Bryant,
\newblock \emph{The supremum of autoconvolutions, with applications to additive
number theory.}
\newblock Illinois J. Math. \textbf{53} (2009), no. 1, 219--235.

\bibitem{MR}
\newblock J. Madrid and J. P. G. Ramos, 
\newblock \emph{On optimal autocorrelation inequalities on the real line.}
\newblock Communications on Pure and Applied Analysis (2021), 20 (1) 369-388



\bibitem{MV} 
\newblock  M. Matolcsi and C. Vinuesa, 
\newblock \emph{Improved bounds on the supremum of autoconvolutions.} 
\newblock Journal of Mathematical Analysis and Applications \textbf{372} (2010), no. 2, 439--447.

\bibitem{Y}
\newblock  G. Yu, 
\newblock \emph{An upper bound for $B2[g]$ sets.}
\newblock Journal of Number Theory \textbf{122} (2007), n. 1, 211--220. 


\end{thebibliography}
